\documentclass[10pt]{article}
\font\smallit=cmti10

\usepackage{amssymb,latexsym,amsmath,epsfig,amsthm,mathptmx,times,float,graphicx,caption,verbatim,url}
\usepackage[latin1]{inputenc}

\makeatletter

\renewcommand\section{\@startsection {section}{1}{\z@}
{-30pt \@plus -1ex \@minus -.2ex}
{2.3ex \@plus.2ex}
{\normalfont\normalsize\bfseries}}

\renewcommand\subsection{\@startsection{subsection}{2}{\z@}
{-3.25ex\@plus -1ex \@minus -.2ex}
{1.5ex \@plus .2ex}
{\normalfont\normalsize\bfseries}}

\renewcommand{\@seccntformat}[1]{\csname the#1\endcsname. }

\makeatother

\begin{document}

\begin{center}
\uppercase{\bf On Generalized Addition Chains}
\vskip 20pt
{\bf Yara Elias\footnote{Supported by a research scholarship of the Fonds qu\'eb\'ecois de la recherche sur la nature et les technologies.}}\\
{\smallit Department of Mathematics, McGill University, Montreal, Quebec, Canada}\\
{\tt yara.elias@mail.mcgill.ca}\\
\vskip 10pt
{\bf Pierre McKenzie\footnote{Supported by DIGITEO and
    by the Natural Sciences
    and Engineering Research Council of Canada.}}\\ 
{\smallit D\'ep.\ d'informatique et de rech.\ op\'erationnelle, U.\ de
  Montr\'eal, Quebec, Canada\\and Chaire Digiteo ENS Cachan-École Polytechnique, France}\\ 
{\tt mckenzie@iro.umontreal.ca}\\ 
\end{center}

\newtheorem{theorem}{Theorem}[section]
\newtheorem{lemma}[theorem]{Lemma}
\newtheorem{corollary}[theorem]{Corollary}
\newtheorem{claim}[theorem]{Theorem}
\newtheorem{introduction}[theorem]{Introduction}
\newtheorem{proposition}[theorem]{Proposition}
\theoremstyle{definition}
\newtheorem{fact}[theorem]{Fact}
\newtheorem{conjecture}[theorem]{Conjecture}
\newtheorem{definition}[theorem]{Definition}
\newtheorem{example}[theorem]{Example}
\newtheorem{remark}[theorem]{Remark}
\newtheorem{question}[theorem]{Question}
\newtheorem{notation}[theorem]{Notation}
\newtheorem{algorithm}[theorem]{Algorithm}
\newcommand{\pgcd}{\mathrm{PGCD}}

\centerline{\bf Abstract}

\noindent
Given integers $d \geq 1,$ and $ g \geq 2$, a $g$-addition chain for $d$ is a
sequence of integers $a_0=1, a_1, a_2,\cdots, a_{r-1}, a_r=d$ where
$a_i=a_{j_1}+a_{j_2}+\cdots +a_{j_k},$ with $2 \leq k \leq g,$ and $0 \leq j_1
\leq j_2 \leq \cdots \leq j_k \leq i-1$. The length of a $g$-addition
chain is $r$, the number of terms following 1 in the sequence.
We denote by $l_g(d)$ the length of a shortest addition chain for
$d$. Many results have been established in the case $g=2$. Our aim is
to establish the same sort of results for arbitrary fixed $g$.
In particular, we adapt methods for constructing
$g$-addition chains when $g=2$ to the case $g>2$ and
we study the asymptotic behavior of $l_g$.

\section{Introduction}
Given integers $d \geq 1,$ and $ g \geq 2$, a $g$-addition chain
for $d$ is a sequence of integers  
\begin{align*}
& a_0=1, a_1, a_2,\cdots, a_{r-1}, a_r=d \end{align*}
where $a_i=a_{j_1}+a_{j_2}+\cdots +a_{j_k},$ with $ 2 \leq k \leq g, \hbox{ and }0 \leq j_1 \leq j_2 \leq \cdots \leq j_k \leq i-1.$
The length of a $g$-addition chain is $r$, the number of terms
following 1 in the
sequence. We denote by $l_g(d)$ the length of a shortest
addition chain for $d$.

Knuth~\cite{Kn69} attributes the first mention of the problem of
determining $l_2(d)$ to H.~Dellac in 1894.
Knuth also reports that
E.~de Jonqui\`eres in 1894 applied what is now known as the factor
method
to the computation of $2$-addition chains.
The term \emph{addition chain} itself,
meaning $2$-addition chain, was coined
and formally defined in 1937 by Scholz~\cite{Sc37}.
While many conjectures (and theorems!) concerning addition chains
rose and fell over the years, the celebrated 1937 Scholz-Brauer
conjecture, claiming that $l_2(2^n-1)\leq n-1+l_2(n)$,
remains open today.

The Scholz-Brauer conjecture and the intriguing behavior of the $l_2$
function led to an abundant literature on addition chains.
Knuth~\cite[Section 4.6.3]{Kn69} is a careful source of facts and historical
details covering the period up to 1973.
Further developments, including world records and a
bibliography reaching until 2008, can be found at \cite{Fl11}.

To the best of our knowledge, none of the above literature considers
$g$-addition chains for $g>2$.
We begin investigating such ``generalized'' addition
chains here.
Specifically,
Section~\ref{sect:construct}
describes three algorithms to generate $g$-addition chains.
In Section~\ref{sect:asymptotic}, we establish upper and lower bounds
on $l_g(d)$ and we bound
the main term and the error term in the
asymptotic behavior of $l_g(d)$.
Section~\ref{sect:conclusion} concludes by
recalling the algebraic complexity theory context in which
the study of addition chains can be cast and by listing
open questions and suggestions for future work.

When $\varsigma$ is a sequence of integers $i_1,\ldots,i_j$ and $m$ is an
integer, we let $m\cdot \varsigma$ stand for
$m\cdot i_1,\ldots,m\cdot i_j$.
We also adopt the following notation:\\
\begin{tabular}{r p{12cm} }
$\lfloor x \rfloor$ & floor of $x$ \\
$\lceil x \rceil$ & ceiling of $x$ \\
$\lambda_g(n)$ & $\lfloor \log_g(n) \rfloor$ \\
$ \mu_g(n)$ & number of nonzero digits in the representation of $n$
in basis $g$ \\
$l_g(n)$ & length of a minimal $g$-addition chain for $n$ \\
$o(1)$ & function $f: \mathbb{N} \rightarrow \mathbb{R}$ such that $f(n)$ goes to 0 when $n$ goes to infinity.
\end{tabular}

\section{Construction of generalized addition chains}\label{sect:construct}

In this section, we extend three methods used to generate $2$-addition chains
for the generation of $g$-addition chains, $g\geq 2$.
We then compare the performances of the methods on selected infinite
families of integers.\\

\newcommand{\fac}{\ensuremath{\mathsf{fac}}}

\subsection{The factor method}

For every $g\geq 2$ and $d\geq 1$,
our extension to the factor method for $2$-addition chains~\cite{Kn69}
produces a unique $g$-addition chain.
This chain is obtained by crossing out duplicates from the sequence $\fac[d]$, defined by induction on $d$ as
\begin{displaymath} 
\begin{cases}
1,\ d & \text{if $d\leq g$,}\\
\fac\left[\frac{d - (d \text{ mod } g)}{g}\right],\ d - (d \text{
    mod } g),\ d & \text{else, if $d$ is prime,}\\
\fac[p_1p_2\cdots p_i],\
(p_1p_2\cdots p_i) \cdot
\fac[p_{i+1}p_{i+2}\cdots p_m]& \text{otherwise}
\end{cases}
\end{displaymath}
where the prime factorization of $d$ is
$p_1p_2\cdots p_m$ with $p_1\leq p_2\leq \cdots \leq p_m$ in the last
case and $i$ is the minimum $j$ such that
$p_1p_2\cdots p_j \geq g$,
unless $j=m$, in which case $i$ is set to $j-1$.

Clearly $\fac[d]$ is well defined.
Note that in the second case,
$d - (d \text{ mod } g)$ is obtained in one step by summing $g$
occurrences of $\frac{d - (d \text{ 
      mod } g)}{g}$; then $d$ is obtained by adding 
$d -  (d \text{ mod } g)$ to $(d \text{ mod } g)$ occurrences of $1$.
In the third case, 
$(p_1p_2\cdots p_i) \cdot
\fac[p_{i+1}p_{i+2}\cdots p_m]$ is obtained by applying the steps defining
the chain for $p_{i+1}p_{i+2}\cdots p_m$ starting from the last number
$p_1p_2\cdots p_i$ of the chain 
obtained for $p_1p_2\cdots p_i$.

When $g=2$, the above method precisely
reduces to the factor method
described in \cite{Kn69}.
We note that the second case in our generalized method exploits the
insight that when $g>2$, merely computing $\fac[d-(g-1)]$ and then $d$
would fail to ensure division by $g$ in the recursive step.
Finally, we note that a possible improvement in the third case would be to
order the prime factors of $d$ in such a way as to
bring $p_1p_2\cdots p_{j}$ closest to $g$.

\begin{example} \label{example1}
Consider $d=(g+1)^2$, where $g+1=p_1^{\alpha_1} \cdots p_k^{\alpha_k}$
is the prime decomposition of $g+1$
with $p_1<\cdots < p_k$.
Assume first that $k>1$.
Then
$d=p_1^{2\alpha_1}\cdots p_k^{2\alpha_k}$.
So the factor method
induces the $g$-addition chain $$
\underbrace{1,\ p_1^{2\alpha_1}\cdots
p_i^{\beta_{i}-1},\ p_1^{2\alpha_1}\cdots p_i^{\beta_i}}_{
\fac[p_1^{2\alpha_1}\cdots p_i^{\beta_i}]}\ ,\ \
p_1^{2\alpha_1}\cdots p_i^{2\alpha_i}\cdots p_k^{2\alpha_k}$$
where $i$ and $0<\beta_i\leq 2 \alpha_i$ are the smallest integers such that
$p_1^{2\alpha_1}\cdots p_i^{\beta_i} \geq g$.
Indeed since $g$ does
not divide $(g+1)^2$, we have
$p_1^{2\alpha_1}\cdots p_i^{\beta_i} \geq g+1$.
Also, since $k>1$, $p_1^{\alpha_1+1}$ divides $p_1^{2\alpha_1}\cdots
p_i^{\beta_i}$, so $p_1^{2\alpha_1}\cdots p_i^{\beta_i} \neq g+1$.
Hence $p_1^{2\alpha_1}\cdots p_i^{\beta_i} \geq g+2$.
Therefore, since $\dfrac{(g+1)^2}{g+2} < g+1$, i.e $\dfrac{(g+1)^2}{g+2} \leq g$, we have $p_i^{2 \alpha_i -\beta_i} \cdots p_k^{2 \alpha_k} \leq g$.
So that the induced addition chain has length 3.
Note that when $g$ is prime, the factor method produces a $g$-addition chain of length at least $e+3$ for $g^e(g+1)^2$. 

In the case $k=1$, the factor method induces the $g$-addition chain $$1,p_k^{\alpha_k-1},p_k^{\alpha_k},p_k^{2\alpha_k-1},p_k^{2\alpha_k} \hbox{ if } \alpha_k>1,$$ and $$1,p_k-1, p_k, p_k(p_k-1),p_k^2 \hbox{ if } \alpha_k=1.$$ 
Note that both are of length 4.
\end{example}
\begin{example} \label{example2}
Consider $d=g^2$, where $g=p_1^{\alpha_1} \cdots p_k^{\alpha_k} $ is
the prime decomposition of $g $ with $p_1< \cdots < p_k$ and assume
that $k>1$. Then the factor method induces the $g$-addition chain $$1,
p_1^{2\alpha_1}\cdots p_i^{\beta_{i-1}},p_1^{2\alpha_1}\cdots
p_i^{\beta_i}, p_1^{2\alpha_1}\cdots p_i^{2\alpha_i}\cdots
p_k^{2\alpha_k}$$ where $i$ and $0<\beta_i \leq 2 \alpha_i$ are the
smallest integers such that $p_1^{2\alpha_1}\cdots p_i^{\beta_i} \geq
g$.
Also, since $k>1$, and thus $p_1^{\alpha_1}<g$, we have that $p_1^{\alpha_1+1}$ divides $p_1^{2\alpha_1}\cdots p_i^{\beta_i}$, so $p_1^{2\alpha_1}\cdots p_i^{\beta_i} \neq g$.
Hence $p_1^{2\alpha_1}\cdots p_i^{\beta_i} > g$.
This addition chain has length 3.

Note that in fact, $d=g^{e+2}$ requires at least $3+e$ steps.
Indeed, the first iteration of the algorithm of the factor method
produces $$1, \fac[q_1], q_1 \cdot \fac\left[\frac{d}{q_1}\right]$$ for some $q_1$
where $g< q_1 \leq g p_k$. Since $q_1>g$, we know that $\fac[q_1]$ contributes at
least $2$ to the
length of the chain. Now applying the
algorithm
to $\dfrac{d}{q_1}$ produces $$1,
\fac[q_2],  q_2 \cdot \fac\left[\frac{d}{q_1q_2}\right],$$ for some $q_2$ where $g<
q_2 \leq g p_k$. Since $q_2>g$, we know that $q_1\cdot \fac[q_2]$ contributes at
least another 2 terms
to the chain. We can repeat
this argument at least $$log_{gp_k}g^{2+e}=\frac{log_g g^{2+e}}{log_g
  gp_k }>\frac{2+e}{2}$$ times, where each time, the length of the
chain increases by 2 at least. Therefore, the
final $g$-addition chain has length at least $3+e$.

When $k=1$, the method induces the $g$-addition chain
$1,p_k^{\alpha_k},p_k^{2\alpha_k}$ of length $2$.
\end{example}

\subsection{The $m$-ary method}

The $m$-ary method consists of expressing $d$ as $ d= d_km^k + \cdots
+ d_1m + d_0$, where $0 \leq d_i < m$ for $0\leq i\leq k=\lfloor
\log_md\rfloor$, and 
  appending to $1$ the sequence
 $$ m, d_km,
  d_km+d_{k-1}, (d_km+d_{k-1})m, d_km^2+d_{k-1}m+ d_{k-2},\cdots,
  (d_km^{k-1}+\cdots +d_1)m,d$$ 
of length at most $\lfloor \log_md\rfloor + \mu_m(d)$ when $m<g$.
When $m\geq g$, the method begins with
$1,d_k \text{ (if $1<d_k<g$)}, g,g+1,g+2,\ldots,m-1$
and appends instead
$$ d_k \cdot \varsigma,$$
$$d_km+d_{k-1},$$
$$ (d_km+d_{k-1}) \cdot \varsigma,$$
$$d_km^2+d_{k-1}m+d_{k-2},$$
$$\vdots$$
$$(d_km^{k-1}+\cdots +d_1)\cdot \varsigma,$$
$$d$$
where $\varsigma$ is a fixed $g$-addition chain for $m$.
Only the digits $d_i$ that are non-zero contribute a
``non-$\varsigma$'' step to the above sequence.
Given an optimal $\varsigma$,
the length of the sequence produced when $m>g$ is thus at most
\begin{eqnarray} \label{eqn:upper}
(m-g+1) + \lfloor \log_md\rfloor l_g(m) + (\mu_m(d)-1).
\end{eqnarray}
Noting that $\ell_g(g^r)=r$ for $r\geq 1$,
the expression (\ref{eqn:upper}) becomes
\begin{eqnarray} \label{eqn:powerupper}
m -g + \lfloor \log_md\rfloor \log_g(m)+ \mu_m(d)
\ \ \leq \ \
m -g + \lfloor \log_gd\rfloor + \mu_m(d)
\end{eqnarray}
in the important special case in which $m$ is a power of $g$.

As finer optimizations,
since adding $d_i<g$ to any
number $A$ can be done from $1$ and $A$ in a single $g$-addition chain
step,
we note that
among the initial $g,g+1,\ldots,m-1$, only numbers that occur as $d_i$
for some $i$ need be produced explicitly.
We note also that expression (\ref{eqn:upper}) can be reduced by $1$
if $d_k=1$ or $d_k\geq g$.

When $g=2$, this method is the same as the $m$-ary method described in
\cite{Kn69}.

\begin{example} \label{example3}
Consider $d=g^k(g+1)^2=g^k(g^2+2g+1)$.
The $g$-ary method induces the following $g$-addition chain, of
length $k+4$:
$$1,g,g+2,g^2+2g,g^2+2g+1,g(g^2+2g+1), \cdots, g^k(g^2+2g+1).$$
\end{example}
\begin{example} \label{example4}
Consider $d=g^{2+k}(2g+1)=2g^{3+k}+g^{2+k}$.
The $g$-ary method induces the following $g$-addition chain, of length
$k+5$:
$$1,2,2g,2g+1,2g^2+g,\cdots,2g^{3+k}+g^{2+k}.$$
\end{example}
Note that multiplying an integer $d$ by $g^k$ extends its
$g$-addition chain obtained by the $g$-ary method by $k$ elements.

\subsection{The tree method}
The tree method consists of drawing a tree, with root 1 and integer
nodes such that the path from the root to the integer $d$ constitutes
a $g$-addition chain for $d$. Let $M_n$ be the set of sums of
$m$-tuples of $\{ 1, a_2, \cdots, a_{k-1}=n \}$, with $2\leq m \leq g$,
where $1, a_2, \cdots, a_{k-1}=n$ is the path from the root to the
node $n$.
At level $k+1$, from left to right, we attach in increasing order,
omitting elements already in the tree, under each element $n$ of the
preceding level $k$, the elements of $M_n$.
When $g=2$, this method is the same as the tree method described in
\cite{Kn69}.

\begin{remark}
In the following example, we solely use the argument that if an integer $d$ is at the level $k$ of the tree, then the integer $gd$ is at worst at the level $k+1$ of the tree.
\end{remark}

\begin{example}\label{example5}
Consider $d=g^2(2g+1)$.
From the tree generated by the tree method, we see that $g$
belongs to level 2, so $2g+1$ belongs to level 3.
Hence $g(2g+1)$ is at worst at level 4, and $g^2(2g+1)$ is at worst at
level 5.
So the length of the induced addition chain is at most 4. 

As the number of steps in the $g$-addition chain for $gn$ using the
tree method is at most the one for $n$ plus one, the tree method
induces a $g$-addition chain of length at most $4+k$ for
$d=g^{2+k}(2g+1)$.
\end{example}

\subsection{Comparison of methods}

Table~\ref{tab:comparisons} summarizes the relative performances of our
three methods on selected families of integers.
The rows in the Table are justified next.
\begin{table}[H] 
\centering
\begin{tabular}{|c|c|c|}
\hline
Compared methods & Property of $g$ & Witness Element/Family \\
\hline
factor  $>$ $g$-ary  & $g+1$ not a power of a prime & $(g+1)^2$ \\
\hline
factor  $>$ $g$-ary  & $g>2$ prime, $g+1$ not a power of 2 & $g^k(g+1)^2$ \\
\hline
factor  $>$ $g$-ary  & $g+1=p^{\alpha}, g>2,\ p$ prime & $2p^{2\alpha}$ \\
\hline
$g$-ary  $>$ factor  && \\
tree  $>$ factor & $g$ not a power of a prime & $g^{2+k}$ \\
\hline
$g$-ary  $>$ factor  && \\
tree  $>$ factor & $g=p^{\alpha},\ p>2$ prime, $\alpha>1$ & $2p^{k \alpha+1}$ \\  
\hline
$g^2$-ary  $>$ factor && \\
tree  $>$ factor & $g$ prime & $(p-1)^2p^{2k}$ \\  
\hline
tree  $>$ $g$-ary  & & $g^{2+k}(2g+1)$ \\

\hline

\end{tabular}
\caption{Comparisons of methods, with ``$A > B$'' shorthand for
  ``method $A$ is strictly more efficient 
  than method $B$''; even when $g=2$, no infinite family seems known
  for which the tree method is systematically outperformed by another
method.} \label{tab:comparisons}
\end{table}

Rows 1 and 2 follow from comparing Examples~\ref{example1}
and~\ref{example3} seen in previous sections;
chain lengths are $3<4$ and $k+3<k+4$ respectively.
For row 3, consider $g=p^{\alpha}-1$, with $p>2$ prime, and $g>2$.
Let $d=2p^{2\alpha}=2g^2+4g+2.$
Then the $g$-ary method induces the $g$-addition chain $$1, 2, 2g,
2g+4,2g^2+4g, 2g^2+4g+2$$ of length 5 while the factor method induces
the shorter $g$-addition chain $1, 2p^{\alpha-1},2p^{\alpha},$ $2p^{2
  \alpha -1},$ $2p^{2 \alpha}$ of length 4.

For rows 4, 5 and 6,
note that the tree method is never worse than the $g$-ary method.
Hence in each row, the second line follows from the first.
Row 4 follows from
the fact that the $g$-ary method induces for $g^{2+k}$, where $k \geq 0$, the
$g$-addition chain $1,g,\cdots,g^{2+k}$ of length $2+k$,
shorter than the chain of length $3+k$ obtained in
Example~\ref{example2} by the factor method.
For row 5, consider $g=p^{\alpha},$ with $p>2$ prime, and $\alpha>1$.
Let $d=2p^{k\alpha+1}$, with $k \geq0$.
The $g$-ary method induces the chain $$1, 2p, 2p p^{\alpha}, \cdots,
2pp^{k\alpha}$$ of length $k+1$, while the factor method induces the
longer chain $$1, 2 p^{\alpha-1},2p^{\alpha},2p^{2\alpha},\cdots,2p^{k
  \alpha},2p^{k \alpha+1}$$ of length $k+2$.
For row 6, let $g=p$, with $p>2$ prime and consider $d=(p-1)^2p^{2k},$ where $ k \geq 0$.
The $p^2$-ary method induces the $g$-addition chain
$$1,p-1,(p-1)^2,p(p-1)^2, p^2(p-1)^2,p^3(p-1)^2,\cdots,p^{2k}(p-1)^2$$
of length $2+2k$.
The factor method induces a longer $g$-addition chain of length at
least $3+2k$.
Indeed, since $p-1$ is even, the first iteration of
the inductive algorithm of the factor method for $d$ produces $2^2q$,
where $q$ is a divisor of $(\frac{p-1}{2})^2$ such that $2^2q\geq
p$.
Now $p$ does not divide $(\frac{p-1}{2})^2$ so $2^2q > p$, therefore the factor method requires two steps to produce $2^2q$.
Also, $q \neq(\frac{p-1}{2})^2$. Indeed, assume $\frac{p-1}{2}$
divides $q$.
Since $2^2 \frac{p-1}{2} >p$, $q$ would have to be equal
to $\frac{p-1}{2}$.
Hence, the $p^2$-ary method produces a $g$-addition chain of length
$2+2k$ for $d=(p-1)^2p^{2k}$, shorter than the one of length at least
$3+2k$ produced by the factor method.

To justify row 7, we combine examples
  \ref{example4} and \ref{example5} and deduce that for each $g$,
  there is an infinite set of integers $d$ of the form
  $g^{2+k}(2g+1)$, where $k \geq 1$, such that the tree method induces a
  $g$-addition chain of length at most $4+k $ shorter than the one by
  the $g$-ary method of length $k+5$.

\subsection{Practical issues}

Suppose that $g\geq 2$ is a fixed integer.
As Theorem~\ref{claim:upper} below makes clear, the $m$-ary method
with $m=g$ implies that
the length of an optimal $g$-addition chain for a number $d$ is no
longer than twice $\log_g(d)$.
Two computational problems thus arise:
\begin{quote}
Given $d$ in binary or decimal notation, compute\\
(1) an optimal $g$-addition chain for $d$\\
(2) a $g$-addition chain for $d$ no longer
than twice the optimal. 
\end{quote}

In complexity theory, efficiency as a first approximation is taken to
mean ``the existence of an algorithm 
that runs in time bounded by some 
polynomial in terms of the problem input length''.
At present, no efficient algorithm is known to
solve problem (1) even when $g=2$.

But we note that problem (2) is solved efficiently by the $m$-ary method
(Sketch: efficient arithmetic to compute the $g$-ary representation
of $d$ from its binary or decimal expansion is well
known~\cite{Kn69}, and a straighforward implementation of the
method involves a polynomial number
of further arithmetic operations.)
On the other hand, the factor method, if it solves problem (2) at all,
is inefficient
because it repeatedly requires factoring numbers
(applied to a number $d$ having all its prime factors larger than $g$,
the method would actually factor $d$ on the fly), for which no
efficient algorithm is currently known.
For its part, the tree method does solve problem (2), but inefficiently
because it potentially examines every 
number less than $d$, hence exponentially many numbers in terms of the
number of digits in the binary or decimal expansion of $d$.

\section{Asymptotic behavior of $l_g(d)$ }\label{sect:asymptotic}

For any $g\geq 2$ and $d\geq 1$, we have $l_g(d) \leq l_2(d) \leq
(g-1)l_g(d)$.
Coarse asymptotic upper bounds on $l_g(d)$ thus follow from
known bounds on $l_2(d)$.
Such coarse bounds vastly overestimate $l_g(d)$ however.
In this section, we provide finer bounds that capture its asymptotic behavior.

Theorem~\ref{claim:upper} and Proposition~\ref{prop:mn}
are straightforward adaptations of the reasoning
for $g=2$.

\begin{claim} \label{claim:upper}
For all $d$ $ \in $ $ \mathbb{N} $,
$$\lceil \log_g d \rceil \leq l_g(d) \leq  \lfloor \log_gd \rfloor
+\mu_g(d) .
$$
\end{claim}
\begin{proof}
Let $d\in \mathbb{N}$.
And let $ a_0=1, a_1, \ldots , a_r=d $ be a $g$-addition chain for $d$ of minimal length $ l_g(d)$.
For all $i$ such that $1\leq i\leq r$, we have $a_i\leq ga_{i-1}$.
Therefore, $d=a_r \leq g^r$, and hence
$ \log_gd \leq \log_g g^r = r = l_g(d)$.
Since $ l_g(d) $ is an integer, $ \lceil \log_gd \rceil \leq l_g(d)$.\\
To establish the upper bound, we use the $g$-ary method (with $m=g$).
We get a $g$-addition chain of length
$$l_g(d)  \leq  \lfloor \log_gd \rfloor +\mu_g(d)$$
as per expression (\ref{eqn:powerupper}).
\end{proof}
\begin{proposition}\label{prop:mn}
For all $m,n$ $ \in $ $ \mathbb{N} $, $l_g(mn) \leq l_g(m)+l_g(n)$.
\end{proposition}
\begin{proof}
A $g$-addition chain for $mn$ is given by a $g$-addition chain for $m$
of length $l_g(m)$ followed by
$m\cdot \varsigma$ where $\varsigma$ is a $g$-addition chain for $n$ of length
$ l_g(n) $.
\end{proof}
The following definition respects the choice of nomenclature in the litterature for $g=2$.
\begin{definition} 
Step $i$ is a \textit{$g$-step} if $a_i = g a_{i-1}$.
\end{definition}

Adapting Brauer and Erd{\H{o}}s' developments in the case $g=2$, we
prove that the asymptotic main term of $l_g(n)$ is larger than $\lambda_g(n)+\dfrac{\lambda_g(n)}{8g log_eg\lambda_g(\lambda_g (n))}$ and smaller than $\lambda_g(n)+\dfrac{\lambda_g(n)}{\lambda_g( \lambda_g (n))}.$

\begin{claim} \label{claimasymptotique}
For all $g \geq 2$, we have $l_g(n) 
 \leq
\lambda_g(n)+(1+o(1))\dfrac{\lambda_g(n)}{\lambda_g( \lambda_g
  (n))}$. (This result is in \cite{Kn69} in the case $g=2$.)
\end{claim}
\begin{proof}
Let $m=g^k$ for any $k\geq 1$. 
Expression ~(\ref{eqn:powerupper}) implies that $$l_g(n) \leq m + \log_g n + \mu_m(n) \leq g^k + (k+1) \log_{g^k}n.$$
So the number of steps is bounded by $$g^k+(k+1)\log_gn
\dfrac{\log_gg}{\log_gg^k}=g^k+\dfrac{k+1}{k}log_gn.$$ 
Let $$k=\lfloor{\lambda_g(\lambda_g(n))-2\lambda_g(\lambda_g(\lambda_g(n)))} \rfloor.$$
Then, 
\begin{align*}
\lambda_g(n) \leq l_g(n)
&\leq g^{\lambda_g(\lambda_g(n))-2\lambda_g(\lambda_g(\lambda_g(n)))}+\left(1+\dfrac{1}{\lfloor\lambda_g(\lambda_g(n))-2\lambda_g(\lambda_g(\lambda_g(n)))\rfloor} \right) \log_g n \\
&\leq \log_g n +\dfrac{g^2\lambda_g(n)}{\lambda_g^2(\lambda_g(n))}+ \dfrac{\log_g n}{ \lfloor\lambda_g(\lambda_g(n))-2\lambda_g(\lambda_g(\lambda_g(n)))\rfloor}.
\end{align*}
We have $\dfrac{\log_g n}{\lfloor \lambda_g(\lambda_g(n))-2\lambda_g(\lambda_g(\lambda_g(n)))\rfloor}=(1+o(1))\dfrac{\lambda_g(n)}{\lambda_g( \lambda_g (n))}$ since $$\displaystyle\lim_{n\to\infty}\dfrac{\log_g n}{\lambda_g (n)}\dfrac{ \lambda_g(\lambda_g(n))}{\lfloor \lambda_g(\lambda_g(n))-2\lambda_g(\lambda_g(\lambda_g(n))\rfloor}-1=0.$$
Also, $\dfrac{g^2\lambda_g(n)}{\lambda_g^2(\lambda_g(n))}=o(1)\dfrac{\lambda_g(n)}{\lambda_g( \lambda_g (n))}$ since $$\displaystyle\lim_{n\to\infty}\dfrac{g^2}{ \lambda_g(\lambda_g(n))}=0.$$
\end{proof}

\begin{corollary} \label{corollaireinfini}
For all $g \geq 2$, we have $\displaystyle\lim_{n\to\infty} \dfrac{l_g(n)}{\lambda_g(n)} 
 =1$. (This result is in \cite{Br39} in the case $g=2$.)
\end{corollary}
\begin{proof}
It is enough to see that $\displaystyle\lim_{n\to\infty} \dfrac{(1+o(1))\lambda_g(n)}{\lambda_g(n)\lambda_g( \lambda_g (n))}=0$.
\end{proof}
Exploiting Erd{\H{o}}s' ideas in the case $g=2$ as in \cite{Er60},
and developing the necessary tools, we show that the main term is larger than $\lambda_g(n)+\dfrac{\lambda_g(n)}{8g\log_e g\lambda_g( \lambda_g (n))}$.
\begin{claim} \label{claimErdos}
Let $g\geq 3$, and let $\varepsilon > 0$. Then,
\begin{align} \label{nombretotal}
\left\vert \left\{ g\mbox{-addition chains } 1=a_0 <\cdots < a_r = n \mbox{ with }\lambda_g(n)=m \mbox{ and } r\leq m +\dfrac{(1-\varepsilon)m}{8g \log_e g \lambda_g(m)} \right\} \right\vert
\end{align}

 $$=\alpha^m$$ for $\alpha<g$ and $m$ large enough. In other words, the number of $g$-addition chains \textit{short enough} is substantially less than $(g-1)g^m$, which is the number of $n$ such that $\lambda_g(n) = m$, for m large enough.
\end{claim}
\begin{proof}
Consider an addition chain $$1=a_0 < \cdots < a_r = n \mbox{ with }\lambda_g(n)=m.$$
Fix a positive integer $K <g$.
Let $A_0$ be the number of $g$-steps in this chain. For such steps, for $i\geq 2$, we have $a_i \leq g^2 a_{i-2}$, and for $i=1$, we have $a_1=ga_0=g$.
For $1 \leq k \leq K$, let $A_k$ be the number of steps $i$ such that $$a_i=(g-k)a_{i-1}+a_{j_1}+\cdots+a_{j_h},$$ $a_{i-1}>a_{j_1} \geq \cdots \geq a_{ j_h},$ $ h \leq k$ and where $g-k$ is the largest coefficient of $a_{i-1}$ among the coefficients of $a_{i-1}$ in the different possible decompositions of $a_i$. For such steps, for $i\geq 2$, $$a_i \leq (g-k)a_{i-1}+ka_{i-2} \leq (g(g-k)+k)a_{i-2}.$$ For $i=1$, $a_i \leq (g-k)a_0=g-k \leq (g(g-k)+k)$.
Finally, let $B$ be the number of steps $i$ such that $a_i=ca_{i-1}+a_{j_1}+\cdots+a_{j_h},$ $c < g-K$, $a_{i-1}>a_{j_1} \geq \cdots \geq a_{ j_h},$ $ h +c\leq g$  and where $c$ is the largest coefficient of $a_{i-1}$ among the coefficients of $a_{i-1}$ in the different possible decompositions of $a_i$. For such steps, for $i\geq 2$, $$a_i < (g-K)a_{i-1}+Ka_{i-2} \leq (g(g-K)+K)a_{i-2}.$$ For $i=1$, we have $a_i < (g-K)a_0=g-K \leq (g(g-K)+K)$.
Now, $r=A_0 + B+\sum_{k=1}^{K}A_k $.\\
We have one possibility for a step accounted for in $A_0$, at most
$r^k$ possibilities (regardless of where the step occurs)
for a step accounted for in $A_k$, and at most $r^g$
possibilities for a step accounted for in $B$.
Hence, 
$$g^{2m} \leq a_r^2 \leq g^{2A_0+1} (g(g-K)+K)^B \prod_{k=1}^{K} (g(g-k)+k)^{A_k}  = gg^{2r} (1-\frac{K}{g}+\frac{K}{g^2})^B \prod_{k=1}^K (1-\frac{k}{g}+\frac{k}{g^2})^{A_k}.$$
Taking logarithm in basis $e$, and using $$\log_e (1-\frac{k}{g}+\frac{k}{g^2})\leq -\dfrac{k}{g}+\dfrac{k}{g^2}=\dfrac{k-gk}{g^2} \hbox{  and  } \log_e(1-\frac{K}{g}+\frac{K}{g^2}) \leq \dfrac{K-gK}{g^2},$$ 
we get
\begin{align} \label{bornesomme}
\dfrac{gK-K}{g^2} B +\sum_{k=1}^K  \dfrac{gk-k}{g^2} A_k \leq 2(r-m+\frac{1}{2}) \log_e g.
\end{align}
\begin{align} \label{un}
\eqref{nombretotal} \leq \sum_{\substack{A_0+B+\sum_{k=1}^K A_k =r \\  \frac{gK-K}{g^2}B+\sum_{k=1}^K  \frac{gk-k}{g^2}A_k  \leq 2(r-m+\frac{1}{2})\log_e g} } \dfrac{r!}{A_0! B! \prod_{k=1}^K A_k! } r^{gB} \prod_{k=1}^K r^{kA_k} .
\end{align} 
The number of terms in the sum \eqref{un} is bounded by $3g(K+1)(r-m+\frac{1}{2})\log_e g$ since $B,A_k$, $k=1, \cdots K$ are bounded by $3g(r-m+\frac{1}{2})\log_e g$.
Also, $$\frac{r!}{A_0!} \leq r^{r-A_0}=r^{B+\sum_{k=1}^K A_k }.$$
Finally, taking into account that $$r^{(K-\frac{K}{g})B+\sum_{k=1}^K (k-\frac{k}{g})A_k} \leq r^{2(r-m+\frac{1}{2})g \log_e g},$$
we obtain:
\begin{align}\label{deux}
\eqref{un} \leq 3g(K+1)(r-m+\frac{1}{2})\log_e g \times r^{B+\sum_{k=1}^K A_k } \times r^{2(r-m+\frac{1}{2})g \log_e g} r^{(g-K+\frac{K}{g})B+\sum_{k=1}^K  \frac{k}{g}A_k}
\end{align}
Choosing $K=\lceil \frac{g}{2} \rceil$ implies
\begin{align*}
6(r-m+\frac{1}{2})g \log_e g \geq 3(K-\frac{K}{g})B+\sum_{k=1}^K 3(k-\frac{k}{g})A_k
\geq (g-K+\frac{K}{g}+1)B + (\frac{k}{g}+1)A_k
\end{align*}
in both cases when $g$ is even or odd.
Therefore, 
\begin{align} \label{trois}
\eqref{deux} \leq 
3(K+1)(r-m)\log_e g \times r^{8(r-m+\frac{1}{2})g \log_e g}
\end{align}
Upon taking $\log_g$ in order to compare with $\log_g((g-1)g^m)=\log_g(g-1)+m$, we get:
\begin{align} \label{quatre}
\log_g \eqref{trois}= \log_g \left( 3g(K+1)(r-m)\log_e g \right) + (8(r-m+\frac{1}{2})g \log_e g)\log_g r
\end{align}
Using $r-m \leq (1- \varepsilon) \dfrac{m}{8g \log_e g \lambda_g(m)}$, and $r\leq 2m$, and letting $m$ go to infinity,
we see that \eqref{quatre} is less than $m$.
\end{proof}

\begin{corollary}
Let $g \geq 3$. For almost all $n$,
$$l_g(n) \geq \lambda_g(n)+\dfrac{\lambda_g(n)}{8g log_eg\lambda_g(\lambda_g (n))},$$
i.e the proportion of integers not satisfying this inequality goes to zero when $n$ goes to infinity.
\end{corollary}

\section{Open questions}\label{sect:conclusion}

Many questions regarding $2$-addition chains remain unsettled. Their
$g$-analogs seem interesting and are at least as hard.

Recall the Scholz-Brauer's conjecture~\cite{Sc37}, concerned with the
worst case behavior
of the $2$-ary method when $g=2$: the conjecture states that for all $n\geq 1$,
$$l_2(2^n-1) \leq n-1 +l_2(n).$$ 
Brauer 
\cite{Br39}
and Hansen 
\cite{Ha59} established a similar inequality, where certain restrictions are imposed on the $2$-addition chain, 
yet the conjecture remains open.
What can we say about $$l_g((g^n-1) + (g + 2g^2 + \cdots + (g-2)g^{g-2}))$$ which seems to be the worst case
for the $g$-ary method?

The conjecture $$l_2(n) \geq \lambda_2(n) + \lceil
\log_2(\mu_2(n))\rceil$$ also remains open, although Sch\"{o}nhage
showed that $$l_2(n) \geq \lceil \log_2(n) + \log_2(\mu_2(n)) -2.13
\rceil$$ in \cite{Sc75}.
Can we prove a similar result for arbitrary $g$?

The functions $d_g(r)= \left\vert \left\{ \mbox { solutions to } l_g(n)=r \right\} \right\vert$, $c_g(r)= \min \left\{n \mid l_g(n)=r \right\}$, as well as 
$NMC_g(n)=\left\vert \left\{g\mbox{-addition chains of minimal length for  }n \right\} \right\vert,$ would be interesting to study; is $d_g(r)$ increasing? How does it evolve asymptotically?
These functions are not well understood, even in the case $g=2$.

Knuth's interest~\cite{Kn69} in addition chains arose from the
fact that $l_2(d)$ is precisely the optimum number of steps required by a
\emph{straight-line $\{\times \}$-program}
computing the univariate polynomial $q(x)=x^d$
out of the initial polynomial $q_0(x)=x$:
\begin{align*}
\text{step $1$:\ \ } q_1 &\leftarrow q_0 \times q_0\\
&\ \ \vdots \\
\text{step $k$:\ \ } q_k &\leftarrow q_{k_1} \times q_{k_2},\ \ \  k_1,k_2<k,\\
&\ \ \vdots \\
\text{step $l_2(d)$:\ \ } q &\leftarrow q_i \times q_j,\ \ \ i,j<l_2(d).
\end{align*}
Obviously, $l_g(d)$ for $g>2$ captures the optimum length of such a
$\{\times\}$-program for $x^d$ in which each step now carries out the
product of up to $g$ factors.
More interestingly,
$\{+,-,\times\}$-programs, in
which the initial polynomials are $1$ and $x$ and a step can now
perform $q_i+q_j$ or $q_i-q_j$ or $q_i\times q_j$, are
a crucial object of study in algebraic complexity theory~\cite{buclsh97}.
A peripheral yet nagging question in that model
has remained open since the 1970's
\cite[p.\ 26]{bomu75}:
does there exist a polynomial $q\in {\mathbb
  Z}[x]$ computable by a $\{+,-,\times\}$-program
that uses \emph{fewer} than $l_2(\text{degree}(q))$ product steps?
The answer at first glance is a resounding ``no'', until one realizes that
cancellation of terms of degree higher than $\text{degree}(q)$ could be helpful.
Such a possibility is tied to the behavior of the $l_2(d)$ function.
The same question now arises in the setting generalized to $g$-ary $\{+,-,\times\}$-programs and $l_g(d)$ for $g>2$.

\section{Acknowledgment}
We would like to thank Andrew Granville for his contribution to the
proof of Theorem \ref{claimErdos} and Michel Boyer for some meticulous
remarks on the document \cite{El11} 
which was 
the starting point of the present paper.
\urlstyle{same}


\begin{thebibliography}{10}

\bibitem{bomu75}
A.~Borodin and I.~Munro.
\newblock {\em The computational complexity of algebraic and numeric problems}.
\newblock Elsevier Computer Science Library, Theory of Computation Series.
  American Elsevier, (1975).

\bibitem{Br39}
A.~Brauer.
\newblock On addition chains.
\newblock {\em Bull. Amer. Math. Soc.}, {\bf 45 }:736--739, (1939).

\bibitem{buclsh97}
P.~Burgisser, M.~Clausen, and A.~Shokrollahi.
\newblock {\em Algebraic Complexity Theory}.
\newblock Grundlehren Math. Wiss. Springer, (1997).

\bibitem{El11}
Y.~Elias.
\newblock Repr\'{e}sentation d'un polyn\^{o}me par un circuit arithm\'{e}tique
  et cha\^{i}nes additives.
\newblock Master's thesis, Universit\'{e} de Montr\'{e}al, (2011).

\bibitem{Er60}
P.~Erd{\H{o}}s.
\newblock Remarks on number theory. {{{{{III}}}}}. {{{{{O}}}}}n addition
  chains.
\newblock {\em Acta Arith.}, { \bf 6}:77--81, (1960).

\bibitem{Fl11}
A.~Flammenkamp.
\newblock { \it
  \url{http://wwwhomes.uni-bielefeld.de/achim/addition_chain.html}}, (2011).

\bibitem{Ha59}
W.~Hansen.
\newblock Zum {{{{{S}}}}}cholz-{{{{{B}}}}}rauerschen problem.
\newblock {\em J.Reine Angew. Math.}, {\bf 202}:129--136, (1959).

\bibitem{Kn69}
D.~E. Knuth.
\newblock {\em The Art of Computer Programming Vol.~II: Fundamental
  Algorithms}.
\newblock Addison-Wesley, (1997).

\bibitem{Sc37}
A.~Scholz.
\newblock Jahresber. dtsch. math.-ver.
\newblock {\em Theot. Comput. Sci.}, { \bf 47 }:41--43, (1937).

\bibitem{Sc75}
A.~Schonhage.
\newblock A lower bound for the length of addition chains.
\newblock {\em Theoret. Comput. Sci.}, {\bf 1}:1--12, (1975).

\end{thebibliography}
\end{document}